\documentclass[12pt,a4paper,reqno]{amsart}

\usepackage{mathtools}
\mathtoolsset{showonlyrefs=true}

\usepackage{amssymb}
\usepackage{amsthm}
\usepackage[all]{xy}
\usepackage{graphicx}

\usepackage{newtxtext}
\usepackage{newtxmath}

\usepackage{geometry}

\newcommand{\OO}{\mathcal O}

\newcommand{\eps}{\varepsilon}

\newcommand{\N}{\mathbb N}
\newcommand{\F}{\mathbb F}
\newcommand{\Hom}{\operatorname{Hom}}

\newcommand{\Ext}{{\operatorname{Ext}}}

\newcommand{\Irr}{\operatorname{Irr}}
\newcommand{\IBr}{\operatorname{IBr}}

\newcommand{\Aut}{\operatorname{Aut}}

\newcommand{\Stab}{{\operatorname{Stab}}}
\newcommand{\mf}{\operatorname{mf}}

\newcommand{\nth}{\operatorname{th}}

\newcommand{\xcdot}{\phantom{\hspace{0.8pt}}}

\newtheorem{defi}{Definition}[section]
\newtheorem{remark}[defi]{Remark}

\newtheorem{conjecture}[defi]{Conjecture}
\newtheorem{thm}[defi]{Theorem}
\newtheorem{lemma}[defi]{Lemma}

\newtheorem{prop}[defi]{Proposition}
\newtheorem*{question}{Question}
\newtheorem*{thm*}{Theorem}
\usepackage{lipsum}

\usepackage[hidelinks]{hyperref}

\title{Arbitrarily large Morita Frobenius numbers}

\author{Florian Eisele}
\address{Department of Mathematics, City, University of London, London EC1V 0HB, United Kingdom
}
\email{florian.eisele@city.ac.uk}

\author{Michael Livesey}
\address{School of Mathematics, University of Manchester, Manchester, M13 9PL, United Kingdom
}
\email{michael.livesey@manchester.ac.uk}

\newtheoremstyle{named}{}{}{\itshape}{}{\bfseries}{.}{.5em}{\thmnote{#3}}
\theoremstyle{named}

\renewcommand{\leq}{\leqslant}
\renewcommand{\geq}{\geqslant}


\begin{document}

\maketitle

\begin{abstract}
	We construct blocks of finite groups with arbitrarily large Morita Frobenius numbers, an invariant which determines the size of the minimal field of definition of the associated basic algebra. This answers a question of Benson and Kessar. This also improves upon a result of the second author where arbitrarily large $\OO$-Morita Frobenius numbers are constructed.
\end{abstract}

\section{Introduction}

Let $\ell$ be a prime and $k$ an algebraically closed field of characteristic $\ell$. 
For a finite-dimensional $k$-algebra $A$ we define the $n^{\nth}$ Frobenius twist of $A$, denoted $A^{(\ell^n)}$, as follows: as a set, and indeed as a ring,  $A^{(\ell^n)}$ is equal to $A$, but for $\lambda \in k$ and $a\in A^{(\ell^n)}$ we set $\lambda\cdot a = \lambda^{\ell^{-n}} a$ (the multiplication on the right hand side being that of $A$). That is, if we think of a $k$-algebra as a ring with a distinguished embedding $k\hookrightarrow Z(A)$, then that embedding is precomposed with the $n^{\nth}$ power of the Frobenius automorphism to obtain the $k$-algebra structure of $A^{(\ell^n)}$. The result of this construction is clearly isomorphic to $A$ as a ring, but not necessarily as a $k$-algebra.
This leads to the following notion, first defined by Kessar~\cite{KessarDon}.
\begin{defi}
	The \emph{Morita Frobenius number} of $A$, denoted $\mf(A)$, is the smallest $n\in \N$ such that $A$ is Morita equivalent to $A^{(\ell^n)}$ as a $k$-algebra.
\end{defi}
As an alternative characterisation, Kessar~\cite{KessarDon} showed that, for a basic algebra $A$, $\mf(A)$ is the smallest $n\in\N$ such that $A\cong k\otimes_{\F_{\ell^n}} A_0$ for some $\F_{\ell^n}$-algebra $A_0$. For fixed $n\in\N$  there are only finitely many possibilities for $A_0$ in any given dimension, which is why Morita Frobenius numbers are being used to approach Donovan's famous finiteness conjecture (more on this further below). 

In the present paper we are interested in the Morita Frobenius numbers of blocks of finite groups $G$. 
For a block $B$ of $kG$, the Frobenius twist $B^{(\ell^n)}$ is isomorphic as a $k$-algebra to $\sigma^n(B)$, where $\sigma$ is the ring automorphism
\begin{equation}\label{algn:ring_auto}
\sigma:\ kG \longrightarrow kG,\ 
\sum_{g\in G}\alpha_g\xcdot g\mapsto\sum_{g\in G}\alpha_g^{\ell}\xcdot g.
\end{equation}
We can therefore think of a Frobenius twist of a block simply as another block of the same group algebra, Galois conjugate to the original one. And while there is no bound on the number of Galois conjugates of a block, Benson~and~Kessar~\cite{BenKesIneq} observed that Morita Frobenius numbers of blocks tend to be very small, with no known example exceeding Morita Frobenius number two. This prompted them to ask the following.
\begin{question}[{Benson-Kessar, \cite[Question 6.2]{BenKesIneq}}]
	Is there a universal bound on the Morita Frobenius numbers of $\ell$-blocks of finite groups?
\end{question} 
This question, to which we give a negative answer in the present article, has gained much interest in recent years.  In~\cite[Examples 5.1, 5.2]{BenKesIneq} Benson and Kessar constructed blocks with Morita Frobenius number two, the first discovered to be greater than one. The relevant blocks all have a normal, abelian defect group and abelian $\ell'$ inertial quotient with a unique isomorphism class of simple modules. It was also proved that amongst such blocks the Morita Frobenius numbers cannot exceed two~\cite[Remark 3.3]{BenKesIneq}. In work of Benson, Kessar and Linckelmann~\cite[Theorem 1.1]{BenKesLinNorm} the bound of two was extended to blocks that don't necessarily have a unique isomorphism class of simple modules. One can also define Morita Frobenius numbers over a complete discrete valuation ring $\OO$ of characteristic zero with residue field $k$, and in \cite{BenKesLinNorm} it was also shown that the aforementioned bound of two applies equally to the Morita Frobenius numbers of the corresponding blocks defined over $\OO$. Finally, Farrell~\cite[Theorem 1.1]{FarrellMFno} and Farrell and Kessar~\cite[Theorem 1.1]{FarrellKessarRational} proved that the Morita Frobenius number of any block of a finite quasi-simple group is at most four (both over $k$ and over $\OO$).

Our main result (see Theorem~\ref{thm:main}) is the $k$-analogue of~\cite[Theorem 3.6]{LiveseyOLarge}, where the corresponding result is proved for blocks defined over $\OO$. Note that the result in the current paper takes significantly longer to prove as Weiss' criterion to detect $\ell$-permutation modules does not hold over $k$.

\begin{thm*}
	For any $n\in\N$ there exists an $\ell$-block $B$ of $kG$, for some finite group $G$, such that $\mf(B)=n$.
\end{thm*}
Hence the questions \cite[Questions 6.2, 6.3]{BenKesIneq} (the second of which is the one mentioned above) both have negative answers.
%	The paper \cite{BenKesIneq} also raises the question whether the Morita Frobenius number of a block can be bounded in terms of the defect group. While we cannot answer this latter question, we should point out that the defect of our  blocks realising Morita Frobenius number $n\in \N$ grows exponentially in $n$. Hence our result would be consistent with a logarithmic bound of Morita Frobenius numbers in terms of defect. 
The blocks realising arbitrarily large Frobenius numbers have elementary abelian defect groups and metabelian $\ell'$ inertial quotients, and turn out to be Morita equivalent to a \emph{twisted tensor product} of the algebra $k[D\rtimes P]$ with itself, where $D$ is an elementary abelian $\ell$-group and $P$ is a cyclic $\ell'$-group. It should be mentioned that while the blocks realising Morita Frobenius number two in \cite{BenKesIneq} are described as \emph{quantum complete intersections}, such algebras can also be realised as iterated twisted tensor products of group algebras of cyclic groups.

	To close, let us quickly explain how the initial motivation to consider Morita Frobenius numbers came from their link with Donovan's conjecture. 
\begin{conjecture}[Donovan]
	Let $D$ be a finite $\ell$-group. Then, amongst all finite groups $G$ and blocks $B$ of $k G$ with defect group isomorphic to $D$, there are only finitely many Morita equivalence classes.
\end{conjecture}
If true, Donovan's conjecture would imply that Morita Frobenius numbers are bounded in terms of a function of the isomorphism class of the defect group.

\begin{conjecture}[{see \cite[Question 6.1]{BenKesIneq}}]\label{con:bdmf}
	Let $D$ be a finite $\ell$-group. Then, amongst all finite groups $G$ and blocks $B$ of $k G$  with defect group isomorphic to $D$, the Morita Frobenius number $\mf(B)$  is bounded.
\end{conjecture}
While we cannot contribute much to this, we should point out that the defect of our  blocks realising Morita Frobenius number $n\in \N$ grows exponentially in $n$. Hence our result does not contradict Conjecture~\ref{con:bdmf} and it would in fact be consistent with a logarithmic bound of Morita Frobenius numbers in terms of the rank of $D$. 
 In~\cite[Theorem 1.4]{KessarDon} Kessar proved that Donovan's conjecture  is equivalent to Conjecture~\ref{con:bdmf} together with the so-called Weak Donovan conjecture, which further highlights the importance of understanding and bounding Morita Frobenius numbers.
\begin{conjecture}[Weak Donovan]
	Let $D$ be a finite $\ell$-group. Then there exists a constant $c(D)\in\N$ such that if $G$ is a finite group and $B$ is a block of $k G$ with defect group isomorphic to $D$, then the entries of the Cartan matrix of $B$ are bounded by  $c(D)$.
\end{conjecture}

%In~\cite[Theorem 1.2]{EatEisLivDon} Eaton and the authors proved that Donovan's conjecture stated over $\OO$ is equivalent to Conjecture~\ref{con:bdmf} stated over $\OO$ together with the Weak Donovan conjecture. That is, both over $k$ and over $\OO$, to go from the Weak Donovan conjecture to Donovan's actual conjecture, all one needs to do is bound the Morita Froebnius numbers. 

The article is organised as follows. In $\S$\ref{sec setup}  we introduce the block $B(\theta)$, which becomes the focus of study for the remainder of the paper. We describe the simple $B(\theta)$-modules in $\S$\ref{sec simples} and in $\S$\ref{sec subalg} we study a certain subalgebra of $B(\theta)$ more closely. We introduce $B(\theta)_0$, a $k$-algebra Morita equivalent to $B(\theta)$, in $\S$\ref{sec basic alg} and prove our main theorem in $\S$\ref{sec Mor equ}.

\subsection*{Notation}
For an $\ell'$-group $G$, $e_\chi\in kG$ will denote the primitive central idempotent corresponding to $\chi\in\Irr(G)$ and $1_G\in\Irr(G)$ will signify the trivial character. We will often use the fact that $\IBr(G)=\Irr(G)$ for such a group $G$. 

For an arbitrary group $G$, a normal subgroup $N\lhd G$ and $\chi\in\Irr(N)$, we set $\Irr(G|\chi)$ to be the set of  characters of $G$ appearing as a non-zero constituent in $\chi\uparrow_N^G$.
% and, for a block $B$ of $kG$, $\Irr(B|\chi)=\Irr(B)\cap\Irr(G|\chi)$. % TODO: we never use the last notation
For $x\in kN$ and $g\in G$, we denote by $x^g=g^{-1}xg\in kN$. Similarly, for $\chi\in\Irr(N)$, we signify by $\chi^g$ the character of $N$ given by $\chi^g(h)=\chi(h^{g^{-1}})$, for all $h\in N$. Note this definition ensures that $e_\chi^g=e_{\chi^g}$. If $\chi,\chi'\in \Irr(N)$ such that $\chi'=\chi^g$, for some $g\in G$, we write $\chi\sim_G \chi'$.

	\section{Setup}\label{sec setup}
	
	In this section we will define the groups and blocks which we later show realise arbitrarily large Morita Frobenius numbers over $k$. All notation introduced in this section will be used throughout the paper. We start by setting, 
	for $i\in \{1,2\}$, 
	\begin{equation}
		P_i=(\F_p,+)=C_p
	\end{equation}
	for some prime $p\neq \ell$, and % TODO: do we need p\equiv 1 mod l? 
	\begin{equation}
	D_i=\prod_{P_i} C_\ell \bigg/ \left\langle (x,\ldots,x) \ | \ x \in C_\ell \right\rangle\cong C^{p-1}_{\ell}.
	\end{equation}
	We set $d_i^x=(1,\ldots,1,d,1,\ldots,1)\in D_i$, where $d$ is a fixed generated of $C_\ell$, and the position of $d$ is the direct factor of $\prod C_\ell$ labeled by $x\in P_i$. In particular, all $d_i^x$ taken together generate $D_i$. The group
	$P_i$ acts on $D_i$ by permuting the direct factors, i.e. by setting $(d_i^x)^y=d_i^{xy}$ for $x,y\in P_i$. Hence we can form the algebra
	\begin{equation}
		A_i=k[D_i\rtimes P_i].
	\end{equation} 
	Let $L=\langle g_0 \rangle \subseteq \F_p^\times \cong C_{p-1}$ be an $\ell'$-subgroup of order $r>1$. Set 
	\begin{equation}
	H=\langle g_1,g_2,g_z : g_1^r=g_2^r=g_z^r=1, [g_1,g_z]=[g_2,g_z]=1, [g_1,g_2]=g_z\rangle
	\end{equation}
	where we adopt the convention that $[g,h]=ghg^{-1}h^{-1}$,
	 and define the subgroups
	\begin{equation} 
		L_1=\langle g_1\rangle,\
	L_2=\langle g_2\rangle\textrm{ and }Z =\langle g_z\rangle.
	\end{equation}
	 We have $L_1\cong L_2\cong Z\cong L\cong C_r$. 
	 Note we have an action of $\F_p^\times$ (and hence $L$) on each $P_i$ given by multiplication.
	 We can now define an action of $H$ on $(D_1\rtimes P_1)\times (D_2\rtimes P_2)$, with kernel $Z$, in the following way. If $\{i,j\}=\{1,2\}$, then $L_i$ acts on $D_i\rtimes P_i$ by setting
	\begin{equation}
	(d_i^xy)^w=d_i^{(x^w)}y^w\quad \textrm{for $x,y\in P_i$ and $w\in L_i$,} 
	\end{equation}
	and setting the action of $L_i$  on $D_j\rtimes P_j$ to be trivial. 
	\begin{defi}[The group $G$, and the block $B(\theta)$]
	Define
	\begin{equation}
		G=((D_1\rtimes P_1)\times (D_2\rtimes P_2))\rtimes H.
	\end{equation}
	and the following subgroups of $G$
	\begin{equation}
		D=D_1\times D_2,\quad\textrm{and}\quad  E=(P_1\times P_2)\rtimes H.
	\end{equation}
	Let $\theta\in \Irr(Z)$ be a faithful character, and let $e_{\theta}$ be the associated central-primitive idempotent. Define a block $B(\theta)=kGe_\theta$.
	\end{defi}

	\begin{defi}\label{def h chi}
	Let $\theta\in \Irr(Z)$ be faithful, as before, and 
	let $\{i,j\}=\{1,2\}$.
	For each $\chi \in \Irr(L_i)$ define an element $h^\theta_{\chi,i}\in L_j$ such that
	\begin{equation}\label{eqn chi theta}
		\chi(-)=\theta([h^\theta_{\chi,i},-])
	\end{equation}
	and 
	\begin{equation}\label{eqn hchi multiplication}
		h^\theta_{\chi,i}h^\theta_{\eta,i}=h^\theta_{\chi\eta,i} \quad\textrm{ for all $\chi,\eta \in \Irr(L_i)$.}
	\end{equation}
	We will often refer to $h^\theta_{\chi,i}$ as $h_{\chi,i}$ where the choice of $\theta$ is clear from the context.
	\end{defi}
	
	Note that in the foregoing definition, the existence of an $h^\theta_{\chi,i}$ satisfying \eqref{eqn chi theta} is guaranteed by \cite[Lemma 4.1]{HolKesQuant} and the uniqueness of such an $h^\theta_{\chi,i}$ in $L_j$ follows by the fact that $C_H(L_i)=Z\times L_i$. 	
	In order to see that \eqref{eqn hchi multiplication} holds we note that, since $[H,H]\subseteq Z \subseteq Z(H)$, we have 
	\begin{equation}
	[h^\theta_{\chi,i}, g]\xcdot [h^\theta_{\eta,i}, g]=h^\theta_{\chi,i} g (h^\theta_{\chi,i})^{-1}g^{-1} [h^\theta_{\eta,i}, g]=h^\theta_{\chi,i} [h^\theta_{\eta,i}, g] g (h^\theta_{\chi,i})^{-1}g^{-1}=[h^\theta_{\chi,i} h^\theta_{\eta,i}, g]
	\end{equation}
	 for all $g\in L_i$. 	
	 Effectively, the above just fixes an isomorphism
	between $\Irr(L_i) \cong \Hom(C_r, k^\times)\cong C_r$ and $L_j \cong C_r$. 
	%	
	%	We also record the following fact which will be important later:
	%
	%	\begin{remark}\label{remark h determine theta}
	%	The  $h^\theta_{\chi,i}$'s from above determine $\theta$, in the sense that if, for $\{i,j\}=\{1,2\}$, $h^\theta_{\chi,i}=h^{\theta'}_{\chi,i}$ for all $\chi\in \Irr(L_i)$, then $\theta=\theta'$. This is due to the fact that $L_j=\{h^\theta_{\chi,i}|\chi\in\Irr(L_i)\}$ and $[L_i,L_j]=Z$, which allows us to recover $\theta$ via equation~\eqref{eqn chi theta}. 
	%	\end{remark}
	%	
	
	\section{Simple modules and Brauer characters}\label{sec simples}
	
	From now on, unless we are explicitly considering $B(\theta)$ and $B(\theta')$ for two $\theta,\theta'\in\Irr(Z)$, we denote $B(\theta)$ simply by $B$. Since $D$ acts trivially on every simple $B$-module, we can and do identify $\IBr(B)$ with $\Irr(E|\theta)$. In what follows, by an abuse of notation, we often use $1$ to denote $1_{P_i}$, for $i=1,2$. We define the following elements of $\Irr(E|\theta)$,
	\begin{equation}\label{algn:lab_char}
		\begin{split}
			(1,1)&=(\theta\otimes 1_{P_1\times (P_2\rtimes L_2)})\uparrow_{Z\times P_1\times(P_2\rtimes L_2)}^E=(\theta\otimes 1_{(P_1\rtimes L_1)\times P_2})\uparrow_{Z\times(P_1\rtimes L_1)\times P_2}^E,\\
			(\phi,1)&=(\theta\otimes \phi\otimes 1_{P_2\rtimes L_2})\uparrow_{Z\times P_1\times (P_2\rtimes L_2)}^E,\\
			(1,\psi)&=(\theta\otimes 1_{P_1\rtimes L_1}\otimes \psi)\uparrow_{Z\times (P_1\rtimes L_1)\times P_2}^E,\\
			(\phi,\psi)&=(\theta\otimes \phi\otimes \psi)\uparrow_{Z\times P_1\times P_2}^E,
		\end{split}
	\end{equation}
	for all $\phi\in\Irr(P_1)\setminus\{1\}$ and $\psi\in\Irr(P_2)\setminus\{1\}$. Note that, since $C_{L_1}(L_2)=C_{L_2}(L_1)=\{1\}$ and $\theta$ is faithful, $\Stab_H(\theta\otimes 1_{L_i})=Z\times L_i$, for all $i=1,2$. Also, as any non-trivial $\phi\in\Irr(P_i)$ is faithful, $\Stab_{L_i}(\phi)=\{1\}$, for all $i=1,2$. It follows that all the characters in \eqref{algn:lab_char} are indeed irreducible.% We will often denote an arbitrary element of $\Irr(E|\theta)$ by $(\phi,\psi)$, for some $\phi\in\Irr(P_1)$ and $\psi\in \Irr(P_2)$.
	
	\begin{lemma}\label{lem:simples}
		$\Irr(E|\theta)=\{(\phi,\psi)|\phi\in\Irr(P_1),\ \psi\in \Irr(P_2)\}$ and $(\phi,\psi)=(\phi',\psi')$ if and only if either $\phi=\phi'$ and $\psi=\psi'$,
		 or $\phi,\psi,\phi',\psi'\neq 1$ and $\phi\sim_{L_1}\phi'$, $\psi\sim_{L_2}\psi'$.
		 Moreover,
		\begin{align*}
			\deg(1,1)=\deg(\phi,1)=\deg(1,\psi)=r, \quad \deg(\phi,\psi)=r^2,
		\end{align*}
		for all $\phi\in\Irr(P_1)\setminus\{1\}$ and $\psi\in\Irr(P_2)\setminus\{1\}$.
	\end{lemma}
	
	\begin{proof}
		We claim that
		\begin{align*}
			\{(1,1)\}&=\Irr(E|\theta\otimes 1_{P_1}\otimes 1_{P_2}),\\
			\{(\phi,1)\}_{\phi\in\Irr(P_1)\setminus\{1\}}&=\bigcup_{\mu\in\Irr(P_1)\setminus\{1\}}\Irr(E|\theta\otimes\mu\otimes 1_{P_2}),\\
			\{(1,\psi)\}_{\psi\in\Irr(P_2)\setminus\{1\}}&=\bigcup_{\nu\in\Irr(P_2)\setminus\{1\}}\Irr(E|\theta\otimes 1_{P_1}\otimes\nu),\\
			\{(\phi,\psi)\}_{\phi\in\Irr(P_1)\setminus\{1\}, \psi\in\Irr(P_2)\setminus\{1\}}&=\bigcup_{\substack{\mu\in\Irr(P_1)\setminus\{1\}\\ \nu\in\Irr(P_2)\setminus\{1\}}} \Irr(E|\theta\otimes \mu\otimes\nu).	
		\end{align*}
		These equalities can all be readily checked. The main point is that, by the comments preceding the lemma, $L_1$ acts regularly on $\Irr(Z\times L_2|\theta)$ and $L_2$ on $\Irr(Z\times L_1|\theta)$. These facts are needed to prove the first three equalities. The fourth is more straightforward. The fact that there are no duplicates, other than the desired ones, is again a consequence of the regularity of these actions. % TODO: There needss to be a lot more detail here. How the group Z\times L_2 come into play at all? Why not use Mackey-decomposition?
		It is a simple task to verify the degrees.
	\end{proof}

    \begin{prop}\label{prop isoms}
	\begin{enumerate}
		\item Let $g_i\in\Aut(P_i) \cong \F_p^\times$ for $i\in\{1,2\}$. The following automorphism of $G$ induces an automorphism of $B(\theta)$,
		\begin{equation}
			\begin{array}{lll}			d_i^xy&\mapsto d_i^{(x^{g_i})}\xcdot y^{g_i},&\text{ for all }i\in\{1,2\}\text{ and }x,y\in P_i,\\
			h&\mapsto h,&\text{ for all }h\in H.
			\end{array}
		\end{equation}
		Furthermore, the corresponding permutation of $\IBr(B(\theta))$ is given by ${(\phi,\psi)\mapsto(\phi^{g_1},\psi^{g_2}})$, for all $\phi\in\Irr(P_1), \psi\in\Irr(P_2)$.
		\item The following automorphism of $G$ induces an isomorphism $B(\theta)\xrightarrow{\sim} B(\theta^{-1})$,
		\begin{equation}
			\begin{array}{lll}		
			(x_1,x_2)&\mapsto(x_2,x_1),&\text{ for all }(x_1,x_2)\in (D_1\rtimes P_1)\times (D_2\rtimes P_2),\\
			z&\mapsto z^{-1},&\text{ for all }z\in Z,\\
			g_i&\mapsto g_j,&\text{  (for $\{i,j\}=\{1,2\}$)}, 
			\end{array}
		\end{equation}
		where $g_1$ and $g_2$ are the generators for $L_1$ and $L_2$ defined in \S\ref{sec setup}.
		Also, for the topmost assignment recall that $P_1$ and $P_2$, as well as $D_1$ and $D_2$, are defined as two copies of the same group, i.e. we may identify them. 
		 Furthermore, the corresponding bijection $\IBr(B(\theta))\longrightarrow\IBr(B(\theta^{-1}))$ is given by $(\phi,\psi)\mapsto(\psi,\phi)$, where we identify $\Irr(P_1)$ and $\Irr(P_2)$.
	\end{enumerate}
\end{prop}

\begin{proof}
	This is all straightforward to check.
\end{proof}
	
	\section{Generators and relations for $A_i=k[D_i\rtimes P_i]$}\label{sec subalg}
	
	Let us now give a description of the (isomorphic) algebras $A_i$ for $i\in\{1,2\}$ in terms of quiver and relations. This description will be used implicitly throughout the remainder of the paper. For the sake of readability, we will use the same notation for the generators of $A_1$ and $A_2$. 
	
	\begin{defi}\label{defi e s}
	 Set 
%	\begin{equation}
%		e_\psi = \frac{1}{|P_i|} \sum_{g\in P_i} \psi(g^{-1})\xcdot g \in k[P_i]\subset A_i \quad\textrm{for $\psi \in \Irr(P_i)$}
%	\end{equation}
%	and 
	\begin{equation}
		s_{\phi} = \sum_{g\in P_i} \phi(g^{-1}) \xcdot d_i^g \in k[D_i] \subset A_i \quad \textrm{for $\phi \in \Irr(P_i)\setminus \{1\}$},
	\end{equation}
	as well as
	\begin{equation}
	s_{\psi,\phi} = e_\psi \xcdot s_{\phi} \in A_i\quad \textrm{for $\psi \in \Irr(P_i)$ and $\phi \in \Irr(P_i)\setminus \{1\}$}.
	\end{equation}
	\end{defi}
	Note that $\Irr(P_i)\setminus \{1\}$ equals the set of constituents of the (multiplicity-free) $k[P_i]$-module $k\otimes_{\F_\ell}D_i$ and hence, by \cite[Proposition 5.2]{EiseleSL2}, of $J(k[D_i])/J^2(k[D_i])$.
	\begin{prop}[{see \cite[Proposition 5.3]{EiseleSL2}}]
	\begin{enumerate}
	\item
	The $e_{\psi}$ for $\psi\in\Irr(P_i)$ form a full set of primitive idempotents in $A_i$.
	\item The $s_{\psi,\phi}$ map to a basis of $J(A_i)/J^2(A_i)$ and $e_{\psi}s_{\psi,\phi}=s_{\psi,\phi}e_{\psi\xcdot \phi}$. That is, the $s_{\psi,\phi}$ correspond to arrows in the quiver of $A_i$.
	\item
	The relations between the arrows are generated by
	\begin{equation}
		s_{\psi,\phi}\xcdot s_{\psi\xcdot \phi,\zeta}=s_{\psi,\zeta}\xcdot s_{\psi\xcdot \zeta, \phi}\quad \textrm{for $\psi \in \Irr(P_i)$ and $\phi,\zeta \in \Irr(P_i)\setminus \{1\}$}
	\end{equation}
	and 
	\begin{equation}\label{eqn l power relation}
		s_{\psi,\phi} \xcdot s_{\psi\xcdot \phi,\phi}\xcdot \cdots\xcdot s_{\psi\xcdot \phi^{\ell-1},\phi}=0\quad \textrm{for $\psi \in \Irr(P_i)$ and $\phi \in \Irr(P_i)\setminus \{1\}$}.
	\end{equation}
	\end{enumerate}
	\end{prop}
	A basis of $J(A_i)$ is given by elements of the form
	\begin{equation}
		s_{\psi,\boldsymbol{\phi}}=s_{\psi,\phi_1}s_{\psi\phi_1,\phi_2}\cdots s_{\psi\phi_1\cdots \phi_{m-1},\phi_m}  \textrm{ where $\psi \in \Irr(P_i)$, $\boldsymbol{\phi}=(\phi_1,\ldots,\phi_m) \in \bigcup_{m=1}^\infty(\Irr(P_i)\setminus \{1\})^m$}.
\	\end{equation}
	To be more precise, we get a basis when we let $\boldsymbol{\phi}$ range over a transversal of
	\begin{equation}
		\bigcup_{m=1}^\infty(\Irr(P_i)\setminus \{1\})^m\big/\operatorname{Sym}_m,
	\end{equation}
	where, in addition, $(\phi_1,\ldots,\phi_m)$ must not involve any element of $\Irr(P_i)\setminus \{1\}$ more than $\ell-1$ times. Let $\mathcal I$ denote the set of all possible values for $\boldsymbol{\phi}$ which give rise to non-zero $s_{\psi, \boldsymbol \phi}$'s, and let $\mathcal I/\sim$ denote equivalence classes of $\boldsymbol \phi$'s that give rise to the same $s_{\psi, \boldsymbol \phi}$ (all of this is independent of $\psi$). We have $|\mathcal I/\sim|=\ell^{p-1}-1$, as $\mathcal I/\sim$ is naturally in bijection with maps $\Irr(P_1)\setminus \{1\}\longrightarrow\{0,1,\ldots,\ell-1\}$ which are not identically zero. 

\section{The algebra $B(\theta)_0$}\label{sec basic alg}
	
	We now need to distinguish between the two sets of generators for $A_1$ and $A_2$ introduced earlier. We will always do this implicitly though, and keep the notation from the previous section. Also, $\theta\in\Irr(Z)$ will denote a fixed faithful character in this section.
	
	\begin{defi} We define the $k$-algebra $B(\theta)_0=C_{B(\theta)}(kHe_\theta)$.
	\end{defi}	
	As with $B(\theta)$, we will usually denote $B(\theta)_0$ simply by $B_0$. Note that, by Lemma \ref{lem:simples},
	\begin{equation}
	|\Irr(H|\theta)|=|\Irr(E|\theta\otimes 1_{P_1}\otimes 1_{P_2})|=|\{(1,1)\}|=1
	\end{equation}
	and so $kHe_{\theta}\cong M_r(k)$. In particular,
	\begin{equation}
	B \cong B_0\otimes_k kHe_\theta\cong B_0\otimes_k M_r(k),
	\end{equation}
	where the first isomorphism is given by multiplication. Naturally this shows that $B$ and $B_0$ are Morita equivalent. Moreover, the dimensions of the simple $B_0$-modules are equal to the dimensions of the corresponding simple $B$-modules divided by $r$.
	 Therefore, since by Lemma~\ref{lem:simples} $\deg(\phi,\psi)=r^2$, for any $\phi\neq 1$ and $\psi\neq 1$, $B_0$ is not basic, but it is sufficiently small for our purposes.
	The structure of the algebra $B_0$ described in Definition~\ref{defi twisted tensor} and Proposition~\ref{prop emb ai}~(1)--(4) below is also known as a \emph{twisted tensor product} of $A_1$ and $A_2$, a notion originally introduced in \cite{TwistedTensorOrig} (see also \cite{TwistedTensorModern} for the special type of twisted tensor product that appears in our context).

%	\begin{defi}
%		Let $\theta\in\Irr(Z)$ be a faithful character, and set $B_0=B(\theta)_0$.
%		\begin{enumerate}
%		\item Let $E_1^1=E_1^1(\theta)\in B_0$ be the idempotent $e_1^{(1)}\xcdot e_\theta \in kGe_\theta$, where $e_{1}^{(1)}\in A_1 \subset kG$.
%		\item Let $F_1^1=F_1^1(\theta)\in B_0$ be the idempotent $e_1^{(2)}\xcdot e_\theta \in kGe_\theta$, where $e_{1}^{(2)}\in A_2 \subset kG$.
%		\item Set $E=E(\theta)=1-E_1^1\in B_0$ and $F=F(\theta)=1-F_1^1\in B_0$.
%		\end{enumerate}
%		(I will get rid of this definiton once everything else is rewritten)
%	\end{defi}
%	Note that the fact that $E_1^1$ and $F_1^1$ lie in $B_0$ technically requires verification (i.e. one needs to check that these elements commute with $L_i$, which is immediate from the Definition~\ref{defi e s}). 

	\begin{defi}\label{defi twisted tensor}
%		Let $\theta\in\Irr(Z)$ be a faithful character, and set $B_0=B(\theta)_0$.
		\begin{enumerate}
			\item Define a linear map 
			\begin{equation}
				\pi = \pi_\theta:\ B_0 \longrightarrow A_1\otimes_k A_2
			\end{equation}
			as the restriction of the linear map $kG\xcdot e_\theta\longrightarrow k[(D_1\rtimes P_1)\times (D_2\rtimes P_2)]\cong A_1\otimes_k A_2$ which sends $n\xcdot h\xcdot e_{\theta}$ to $n$ for any $n\in (D_1\rtimes P_1)\times (D_2\rtimes P_2)$ and $h\in L_1\cdot L_2\subset H$ (note that $L_1\cdot L_2$ is not a group).
			\item For $i\in \{1,2\}$ let
			\begin{equation}
				A_i = \bigoplus_{\chi \in \Irr (L_i)} A_i^\chi
			\end{equation}
			be the decomposition of $A_i$ as an $L_i$-module into isotypical components, i.e. $a^g=\chi(g)\xcdot a$ whenever $a\in A_i^\chi$ and $g\in L_i$. We refer to the elements of 
			any one of the spaces $A_i^\chi$ as \emph{homogeneous}.
			\item For $i\in \{1,2\}$ define the linear map $\iota_i=\iota_{i,\theta}$ as follows:
			\begin{equation}
				\iota_i:\ A_i \longrightarrow B_0,\ a \mapsto a\xcdot h_{\chi,i}^{-1} \xcdot e_\theta \quad \textrm{for all $a\in A_i^\chi$ and $\chi\in \Irr(L_i)$.}
			\end{equation}
		\end{enumerate}
	\end{defi}

%	\begin{remark} Note that 
%		$\pi(\iota_1(a))=a\otimes 1$ and $\pi(\iota_2(b))=1\otimes b$
%	for any $a\in A_1$ and $b\in A_2$.
%	\end{remark}
	\begin{remark}
		We will often use without further mention that $e_1	\in A_i^1$ for $i\in\{1,2\}$. Analogous statements are not true for the other idempotents.
	\end{remark}
	
	The next proposition summarises the properties of the maps $\pi$, $\iota_1$ and $\iota_2$, which relate the structure of $B_0$ to that of $A_1\otimes_k A_2$, which we understand completely by \S\ref{sec subalg}.  The $\iota_i$  turn out to be actual algebra homomorphisms. The map $\pi$ induces a bijection between $B_0$ and $A_1\otimes_k A_2$. And while $\pi$ is not an algebra isomorphism, it nevertheless shares some of the properties of an algebra isomorphism (e.g. point~\eqref{point radical} of the proposition below would be obvious if $\pi$ were a isomorphism). 	
	
	\begin{prop}\label{prop emb ai}
%		Let $\theta\in\Irr(Z)$ be a faithful character, and set $B_0=B(\theta)_0$.
		\begin{enumerate}
		\item The map $\iota_i:\ A_i \hookrightarrow B_0$
		is a $k$-algebra homomorphism for $i\in\{1,2\}$.
		\item \label{point commutation}	For $\chi \in \Irr(L_1)$ and $\eta\in\Irr(L_2)$ we have
		\begin{equation}\label{eqn commutation new}
		\iota_1(a) \xcdot \iota_2(b) = \theta([h_{\eta,2}, h_{\chi,1}])\xcdot \iota_2(b) \xcdot \iota_1(a) \quad \textrm{for all $a\in A_1^\chi, b\in A_2^\eta$}.
		\end{equation}
		\item The map $\pi:\ B_0\longrightarrow A_1\otimes_k A_2$ is bijective. 
		\item\label{point formula pi}
		For all $a\in A_1$ and $b\in A_2$ we have 
		\begin{equation}\label{eqn pi prod homogeneous}
		\pi(\iota_1(a)\xcdot \iota_2(b)) =  a\otimes b.
		\end{equation}
		\item \label{point radical} For all $i\geq 1$ we have 
		\begin{equation}
		 \pi(J^i(B_0))=J^i(A_1\otimes_k A_2)=\sum_{j=1}^i J^j(A_1)\otimes_k J^{i-j}(A_2).
		 \end{equation}
%		  For $\chi\in\Irr(L_1)$ and $\eta\in \Irr(L_2)$ we have 
%		 \begin{equation}
%		 \pi(\iota_1(A_1^\chi)\xcdot \iota_2(A_2^\eta)) = A_1^\chi\otimes_k A_2^\eta
%		 \end{equation}
%		 and, more precisely,
%		 \begin{equation}\label{eqn pi prod homogeneous}
%		 \pi(\iota_1(a)\xcdot \iota_2(b)) = \theta([h_{\eta,2},h_{\chi,1}]) \xcdot  a\otimes b
%		 \end{equation}
%		 for all $a\in A_1^\chi$ and $b\in A_2^\eta$.
		 \item\label{point ideals} Let $a_1,a_2\in A_1$ and $b_1,b_2 \in A_2$ be homogeneous elements. Then
		 \begin{equation}
		 	\pi( \iota_1(a_1)\iota_2(b_1) \xcdot B_0 \xcdot \iota_1(a_2)\iota_2(b_2)) = (a_1\otimes b_1) \xcdot (A_1\otimes_k A_2) \xcdot (a_2\otimes b_2).
		 \end{equation}
%		 \item\label{point iso ai} Let $\{i,j\}=\{1,2\}$. The composition of $\iota_i$ with the natural surjection 
%		 \begin{equation}
%		 B_0\longrightarrow B_0/B_0\xcdot \iota_j(1-e_1)\xcdot B_0
%		 \end{equation} 
%		 is an isomorphism. 
		\end{enumerate}
	\end{prop}
	\begin{proof}
		\begin{enumerate}
		\item For $a\in A_i^\chi$ and $b\in A_i^\eta$, we have 
		 \begin{equation}
		 \iota_i(a)\xcdot\iota_i(b)= a\xcdot h_{\chi,i}^{-1} \xcdot e_{\theta} \xcdot b \xcdot h_{\eta,i}^{-1} \xcdot e_\theta = ab\xcdot h_{\chi,i}^{-1}\xcdot h_{\eta,i}^{-1} \xcdot e_\theta = ab\xcdot h_{\chi\eta,i}^{-1} \xcdot e_\theta = \iota_i(ab), 	
		 \end{equation}
		using that $h_{\chi,i} h_{\eta,i}=h_{\chi\eta,i}$, as we saw earlier. The above shows that $\iota_i$ is a $k$-algebra homomorphism, since the various $A_i^\chi$ span $A_i$. 	
		\item We have
		\begin{equation}
	\begin{array}{rcl}
\iota_1(a)\xcdot\iota_2(b)&=& a\xcdot h_{\chi,1}^{-1} \xcdot e_{\theta} \xcdot b \xcdot h_{\eta,2}^{-1} \xcdot e_\theta = a\xcdot b^{h_{\chi,1}}\xcdot h_{\chi,1}^{-1}\xcdot h_{\eta,2}^{-1} \xcdot e_\theta\\ &=& 
\eta({h_{\chi,1}})\xcdot  b\xcdot a\xcdot h_{\eta,2}^{-1}\xcdot h_{\chi,1}^{-1}\xcdot  \theta([h_{\eta,2},h_{\chi,1}^{-1}]) \xcdot e_\theta=  b\xcdot a\xcdot h_{\eta,2}^{-1}\xcdot h_{\chi,1}^{-1} \xcdot e_\theta\\
&=&  \chi(h_{\eta,2}^{-1})\xcdot b\xcdot h_{\eta,2}^{-1} \xcdot a \xcdot  h_{\chi,1}^{-1} \xcdot e_\theta = \theta([h_{\eta,2}, h_{\chi,1}])\xcdot \iota_2(b)\xcdot \iota_1(a).
\end{array}
\end{equation}
		\item Surjectivity of $\pi$ will follow immediately from point \eqref{point formula pi} below. For injectivity we compare dimensions. The image of $\pi$ has dimension $\dim(A_1)\xcdot \dim(A_2)=|D_1|^2\xcdot |P_1|^2$.  On the other hand ${\dim(C_{B(\theta)}(H))\xcdot \dim(kHe_\theta)= \dim(C_{B(\theta)}(H))\xcdot r^2=\dim(B)=|D_1|^2|P_1|^2r^2}$, which shows that $\dim(B_0)=\dim(C_{B(\theta)}(H))= |D_1|^2|P_1|^2$, which is the same as the dimension of the image of $\pi$.
		\item It suffices to check formula~\eqref{eqn pi prod homogeneous} for $a$ and $b$ homogeneous. So assume $a\in A_1^\chi$ and $b\in A_2^\eta$. As in the proof of \eqref{point commutation} we have		
		$
		\iota_1(a)\xcdot\iota_2(b)%= a\xcdot h_{\chi,1}^{-1} \xcdot e_{\theta} \xcdot b \xcdot h_{\eta,2}^{-1} \xcdot e_\theta 
		= a\xcdot b\xcdot h_{\eta,2}^{-1} \xcdot h_{\chi,1}^{-1}\xcdot e_\theta,
		$
		%where we used the fact that $e_\theta$ is central and $b \xcdot h_{\eta,2}^{-1}$ commutes with elements of $L_2$ by construction.
		Since, slightly counter-intuitively, $h_{\chi,1}\in L_2$ and $h_{\eta,2}\in L_1$, we have that $h_{\eta,2}^{-1} \xcdot h_{\chi,1}^{-1}\in L_1\xcdot L_2$.
		Therefore, by definition, $\pi$ maps the above element to $a\otimes b$.	
		\item Recall that $\pi$ was defined as the restriction of a linear map $\hat \pi:\ kG\xcdot e_\theta\longrightarrow k[(D_1\rtimes P_1)\times (D_2\rtimes P_2)]\cong A_1\otimes_k A_2$ which is a homomorphism of left $k[D_1\times D_2]$-modules. In particular, $\hat \pi$ will map $J^i(kG\xcdot e_\theta)$ onto $J^i(A_1\otimes A_2)$ for all $i\geq 0$. 
		This uses that $D_1\times D_2$ is a normal Sylow $\ell$-subgroup of $G$, and therefore $J(kG)=J(k[D_1\times D_2])\xcdot kG$ and an analogous expression for $J(A_1\otimes A_2)$.
		Now $J^i(kGe_\theta) = J^i(B_0)\xcdot kHe_\theta$, and by the definition of $\hat \pi$ we have 
		$\hat \pi(J^i(B_0)\xcdot kHe_\theta) = \pi(J^i(B_0))$, which proves the claim. 
		\item It suffices to check that 
		\begin{equation}
\pi( \iota_1(a_1)\iota_2(b_1) \xcdot (\iota_1(A_1^\chi)\iota_2(A_2^\eta)) \xcdot \iota_1(a_2)\iota_2(b_2)) = (a_1\otimes b_1) \xcdot (A_1^\chi\otimes_k A_2^\eta )\xcdot (a_2\otimes b_2)
		\end{equation}
		for all $\chi\in \Irr(L_1)$ and $\eta\in\Irr(L_2)$.
		However, by formula~\eqref{eqn commutation new}, the relevant factors in the argument of $\pi$ above commute up to a non-zero scalar (which does not affect the image).  Hence the left hand side of the above is equal to
		$
				\pi(\iota_1(a_1A_1^\chi a_2)\xcdot \iota_2(b_1A_2^\eta b_2)),
		$ 
		which equals the right hand side of the above by point~\eqref{point formula pi}.
		\qedhere
%		\item For simplicity assume $i=1$ and $j=2$. Since $1-e_1\in A_2^1$, it follows from point \eqref{point ideals}  that 
%		\begin{equation}
%		\pi(B_0\xcdot \iota_2(1-e_1)\xcdot B_0) = A_1\otimes_k (A_2 \xcdot (1-e_1) \xcdot A_2),
%		\end{equation}
%		where we also used that $\pi(\iota_2(1-e_1))=1\otimes (1-e_1)$. Now the image of $\pi\circ \iota_1$ is equal to 
%		$A_1\otimes_k k\xcdot 1_{A_2}$. Clearly
%		\begin{equation}
%		A_1\otimes_k k\xcdot 1_{A_2} + A_1\otimes_k (A_2 \xcdot (1-e_1) \xcdot A_2) = A_1\otimes_k A_2,
%		\end{equation}
%		which proves surjectivity of the asserted isomorphism. Injectivity follows by comparing dimensions, since $ A_2/(A_2 \xcdot (1-e_1) \xcdot A_2)\cong k$ is one-dimensional, implying that $A_1\otimes_k (A_2 \xcdot (1-e_1) \xcdot A_2)$ has co-dimension $\dim_k(A_1)$ in $A_1\otimes_k A_2$.
%		\qedhere
		\end{enumerate}
	\end{proof}

	\begin{prop}[Explicit formula for $\iota_i$]\label{prop formula iota}
		For $\{i,j\}=\{1,2\}$ we have
		\begin{equation}\label{eqn sjjsjjssj}
			\iota_i(a) = \sum_{g\in L_i} (a\xcdot e_{1_{L_j}} \xcdot e_\theta)^g\quad \textrm{for all $a\in A_i=k[D_i\rtimes P_i]$.}
		\end{equation}
	\end{prop}
	\begin{proof}
		Assume without loss of generality that $i=1$ and $j=2$.
		It suffices to prove this for $a\in A_1^\chi$ for a fixed $\chi\in \Irr(L_1)$. Note the idempotent $e_{1_{L_2}}$ can also be written as $r^{-1}\xcdot \sum_{\eta\in \Irr(L_1)} h_{\eta,1}$. Now by character orthogonality (also using $h_{\eta,1} =[h_{\eta,1}, g^{-1}] \xcdot  h_{\eta,1}^g$),
		\begin{equation}
			\pi\left(\sum_{g\in L_1} a^g\xcdot h_{\eta,1}^g \xcdot e_{\theta}\right)= \pi\left(\sum_{g\in L_1} \chi(g)\eta(g)\xcdot a \xcdot h_{\eta,1} \xcdot e_\theta\right) = \left\{\begin{array}{ll} 0 & \textrm{if $\chi\neq \eta^{-1}$}\\ r\xcdot a\otimes 1 & \textrm{if $\chi = \eta^{-1}$} \end{array},\right.
		\end{equation}
		for all $\eta\in\Irr(L_1)$. As $\pi(\iota_1(a))= a\otimes 1$, summing over all such $\eta$ gives that $\pi$ applied to both sides of \eqref{eqn sjjsjjssj} holds true. Since $\pi$ is bijective, the result follows.
	\end{proof}	
	
	By the above, the algebra $B_0=B(\theta)_0$ can be thought of as being graded by $\Irr(P_1)\times \Irr(P_2)$, and the character $\theta$ determines how the homogeneous components commute through equation~\eqref{eqn commutation new}. We will ultimately recover $\theta$ from $B_0$ by showing that certain subspaces of the homogeneous components are preserved under isomorphisms modulo $J^3(B_0)$. However, we will proceed in a more elementary way, and for that we will need idempotents and certain arrows explicitly.  

\begin{defi} \label{defi eps s}
	\begin{enumerate}
	\item For $\phi \in \Irr(P_1),\psi\in\Irr(P_2)$ set
	\begin{equation}
	\eps_{(\phi,1)}=\iota_1(e_{\phi})\iota_2(e_1)\quad\textrm{and}\quad 
	\eps_{(1,\psi)}=\iota_1(e_1)\iota_2(e_\psi).
	\end{equation}
	For $\phi\in\Irr(P_1)\setminus\{1\}, \psi\in \Irr(P_2)\setminus\{1\}$
	set
	\begin{equation}
	\eps_{(\phi,\psi)} =\iota_1(e_{[\phi]}) \iota_2(e_{[\psi]}),\quad\textrm{where}\ e_{[\psi]}= \sum_{g\in L_i} e_{\psi^g}\in A_i^1.
	\end{equation}
	\item For $\psi,\phi\in \Irr(P_1)$, $\xi,\zeta\in \Irr(P_2)$ such that $\phi\neq 1$, $\zeta\neq 1$ set
	\begin{equation}
		S_{\psi,\phi} = \iota_1(s_{\psi,\phi})\iota_2(e_1) \quad \textrm{ and }\quad T_{\xi,\zeta} = \iota_1(e_1)\iota_2(s_{\xi,\zeta}).
	\end{equation}
	\item For $1\neq \phi\in \Irr(P_1)$, $1\neq \zeta\in \Irr(P_2)$, $\chi\in \Irr(L_1)$ and $\eta\in \Irr(L_2)$ set
	\begin{equation}
		\tilde S_\phi^\chi = \sum_{g\in L_1} \chi(g^{-1}) S_{1,\phi^g}S_{\phi^g, (\phi^{-1})^g}\quad\textrm{and}\quad
		\tilde T_\zeta^\eta = \sum_{g\in L_2} \eta(g^{-1}) T_{1,\zeta^g}T_{\zeta^g, (\zeta^{-1})^g}.
	\end{equation}
	\end{enumerate}
\end{defi}

	\begin{remark}\label{remark iota hom}
	Note that by Proposition~\ref{prop emb ai}~\eqref{point commutation}, for $\{i,j\}=\{1,2\}$, the element $\iota_j(e_1)$, and more generally every element of $\iota_j(A_j^1)$, commutes with every element of $\iota_i(A_i)$. 
	%	Hence
	%	\begin{equation}
	%	A_i\longrightarrow B_0(\theta):\ 
	%	a \mapsto\iota_i(a)\xcdot \iota_j(e_1)
	%	\end{equation} is a (non-unital) algebra homomorphism. 
	It follows that the $\eps_{(\psi,\phi)}$ are idempotents, and 
	\begin{equation}\label{eqn s arrows}
	S_{\psi,\phi} = \eps_{(\psi,1)}\xcdot S_{\psi,\phi}\xcdot \eps_{(\psi\phi,1)} \quad\textrm{and}\quad  	
	T_{\xi,\zeta} = \eps_{(\xi,1)}\xcdot T_{\xi,\zeta}\xcdot \eps_{(\xi\zeta,1)}.		
	\end{equation} 
	Moreover, it follows that 
	\begin{equation}
	S_{\psi_1,\phi_1}\cdots S_{\psi_m,\phi_m} = \iota_1(s_{\psi_1,\phi_1}\cdots s_{\psi_m,\phi_m})\iota_2(e_1)\quad\textrm{	for any $m> 0$,}
	\end{equation}
	\begin{equation}
		S_{\psi_1,\phi_1}\cdots S_{\psi_m,\phi_m} T_{\xi_1,\zeta_1}\cdots T_{\xi_n,\zeta_n} =
		\iota_1(s_{\psi_1,\phi_1}\cdots s_{\psi_m,\phi_m}e_1)\iota_2(e_1s_{\xi_1,\zeta_1}\cdots s_{\xi_n,\zeta_n})\quad\textrm{for any $m,n> 0$,}
	\end{equation}
	for any admissible choice of $\psi_i$, $\phi_i$, $\xi_i$ and $\zeta_i$. This will be useful since the image of the right hand side under $\pi$ can readily be determined using Proposition~\ref{prop emb ai}~\eqref{point formula pi}.
\end{remark}

\begin{prop}\label{prop idempotents}
		If $(\psi,\phi)$ and $(\psi',\phi')$ label two different Brauer characters of $B$ (using the notation of \S\ref{sec simples}), then $\eps_{(\phi,\psi)}$ and $\eps_{(\phi',\psi')}$  are distinct orthogonal idempotents. The idempotent $\eps_{(\phi,\psi)}$ annihilates all simple $B_0$-modules except for the one corresponding to the character $(\phi,\psi)$. In particular, the idempotents $\eps_{(\psi,1)}$ and $\eps_{(1,\phi)}$ are primitive in $B_0$ for any $\psi\in\Irr(P_1), \phi\in\Irr(P_2)$.
\end{prop}
\begin{proof}
			Let $\phi\in\Irr(P_1)$ and $\psi\in\Irr(P_2)$. Using Proposition~\ref{prop formula iota}  and the fact that $\iota_2(e_1)=e_{1_{P_2}}$ (immediately from the definition~of~$\iota_2$) we get 
	\begin{equation}
	\eps_{(\phi,1)} = \iota_1(e_\phi) \xcdot \iota_2(e_1)=
	\sum_{g\in L_1} (e_\phi \xcdot e_{1_{L_2}} \xcdot e_{\theta})^g \xcdot e_{1_{P_2}} = \sum_{g\in L_1} (e_\phi \xcdot e_{1_{P_2\rtimes L_2}} \xcdot e_{\theta})^g.
	\end{equation}	
	The idempotent on the right hand side is clearly the central-primitive idempotent in $kE$ which belongs to the induced module given in equation~\eqref{algn:lab_char}. That is, $\eps_{(\phi,1)}$ is the idempotent attached to the character $(\phi,1)$.  By the same argument $\eps_{(1,\psi)}$ belongs to the character $(1,\psi)$. If $\phi\neq 1$, $\psi\neq 1$ then, using only the definition of $\iota_i$ this time,
	\begin{equation}
	\eps_{(\phi,\psi)}=\iota_1(e_{[\phi]})\xcdot \iota_2(e_{[\psi]}) = \sum_{g\in L_1}e_{\phi^g}\xcdot e_\theta \xcdot \sum_{h\in L_2}  e_{\psi^h}  \xcdot e_\theta = \sum_{g\in L_1} \sum_{h\in L_2} (e_{\phi} \xcdot e_{\psi}\xcdot e_\theta)^{gh},
	\end{equation}
	which is clearly the central-primitive idempotent in $kE$ which belongs to the character $(\phi,\psi)$. 
	The distinctness and orthogonality of the various $\eps_{(\phi,\psi)}$'s follows immediately from the fact that these are central-primitive idempotents belonging to distinct simple modules (albeit in the algebra $kE$). As each simple $kE$-module gives rise to a simple $kG$-module, whose restriction to $B_0$ is a direct sum of copies of a single simple module (since $B_0$ and $B$ are naturally Morita equivalent), the claim regarding the action of these idempotents on the simple $B_0$-modules follows as well. It is also clear that the $\eps_{(\phi,1)}$ and $\eps_{(1,\psi)}$ are primitive, as the corresponding simple $B_0$-modules are one-dimensional. 
\end{proof}

\begin{lemma}\label{lemma arrow properties}
	 \begin{enumerate}
		\item\label{point ext} For all $\psi,\phi\in\Irr(P_1)$ and $\xi,\zeta\in\Irr(P_2	)$ with $\phi\neq 1$ and $\zeta\neq 1$,
		\begin{eqnarray}
		\eps_{(\psi,1)}(J(B_0)/J^2(B_0))\eps_{(\psi\phi,1)}&=&\langle S_{\psi,\phi}\rangle_k+J^2(B_0),\\
		\eps_{(1,\xi)}(J(B_0)/J^2(B_0))\eps_{(1,\xi\zeta)}&=&\langle T_{\xi,\zeta}\rangle_k+J^2(B_0),
		\end{eqnarray}
		and all of these spaces are one-dimensional. That is, the $S_{\psi,\phi}$ and $T_{\xi,\zeta}$ correspond to arrows in the quiver of $B_0$. Moreover,
		\begin{equation}
			\eps_{(\psi,1)}(J(B_0)/J^2(B_0))\eps_{(\psi,1)} =0
		\quad \textrm{and}\quad
		\eps_{(1,\xi)}(J(B_0)/J^2(B_0))\eps_{(1,\xi)}=0.
		\end{equation}
		\item\label{point all psi equal} Let $\phi_1,\dots,\phi_\ell\in\Irr(P_1)\setminus\{1\}$ such that $\phi_1\xcdot\ldots\xcdot\phi_\ell=1$. Then
		\begin{equation}
		S_{1,\phi_1}S_{\phi_1,\phi_2}\cdots S_{\prod_{i=1}^{\ell-1}\phi_i,\phi_\ell}\in J^{\ell+1}(B_0)
		\end{equation}
		if and only if $\phi_1=\dots=\phi_\ell$. An analogous statement holds for the $T_{\xi,\zeta}$'s.
		\item\label{point lin indep} The sets
			\begin{equation}
			\left\{S_{1,\phi}S_{\phi,\phi^{-1}}\right\}_{\phi\in\Irr(P_1)\setminus\{1\}} \quad\textrm{and}\quad 	\left\{S_{1,\phi}S_{\phi,\phi^{-1}}T_{1,\zeta}T_{\zeta,\zeta^{-1}}\right\}_{\phi\in\Irr(P_1)\setminus\{1\}, \zeta\in\Irr(P_2)\setminus\{1\}}
			\end{equation}
			are linearly independent modulo $J^3(B_0)$ and  $J^5(B_0)$, respectively. An analogous statement to the first one holds for the $T_{\xi,\zeta}$'s. 
		\item\label{point schi comm} For any $1\neq \phi\in \Irr(P_1)$, $1\neq \zeta\in \Irr(P_2)$, $\chi\in \Irr(L_1)$ and $\eta\in \Irr(L_2)$ 
		\begin{equation}
			\tilde S^\chi_{\phi} \tilde T^\eta_{\zeta} = \theta([h_{\eta,2}, h_{\chi,1}]) \xcdot \tilde T^\eta_{\zeta} \tilde S^\chi_{\phi}.
		\end{equation}
		\end{enumerate}
\end{lemma}
\begin{proof}
	\begin{enumerate}
		\item 
		First of all note that ``$\supseteq$'' is clear by equation~\eqref{eqn s arrows}.  By Proposition~\ref{prop emb ai}~\eqref{point formula pi} it follows that $\pi(S_{\psi,\phi})=s_{\psi,\phi}\otimes e_1$, and this element is not contained in $J^2(A_1\otimes_k A_2)$. Hence $S_{\psi,\phi}\not\in J^2(B_0)$ by Proposition~\ref{prop emb ai}~\eqref{point radical}. That is, the spaces on the right hand side are all one-dimensional. 
		
		By definition we have 
		\begin{equation}
		\sum_{\psi\in\Irr(P_1)} \eps_{(\psi,1)} = \iota_2(e_1).
		\end{equation} 
		Since the other inclusion is already known, to prove ``$\subseteq$'' it will suffice to show that $\iota_2(e_1) (J(B_0)/J^2(B_0)) \iota_2(e_1)$ (which contains all of the $\eps_{(\psi,1)}(J(B_0)/J^2(B_0))\eps_{(\psi\phi,1)}$)) is spanned by elements of the form $S_{\psi,\phi} +J^2(B_0)$.	Since $\pi$ is bijective we may as well consider the images under $\pi$. Using Proposition~\ref{prop emb ai}~\eqref{point radical}~and~\eqref{point ideals}, as well as $\pi(\iota_2(e_1))=1\otimes e_1$, we have 
		 	\begin{equation}
		 	    \begin{array}{rl}
		 	    &\pi( \iota_2(e_1)(J(B_0)/J^2(B_0))\iota_2(e_1))\\
		 	    =&\left[J(A_1)/J^2(A_1)\otimes_k e_1(A_2/J(A_2))e_1\right] \oplus \left[A_1/J(A_1)\otimes_k  e_1(J(A_2)/J^2(A_2))e_1\right],
		 	    \end{array}
		 	\end{equation}
		 	which is spanned by elements of the form $s_{\psi,\phi}\otimes e_1 + J^2(A_1\otimes_k A_2)=\pi(S_{\psi,\phi} +J^2(B_0))$, since the second bracket is zero. This proves the first claim. 
		 	
		 	The second claim follows from the fact that, just like the other spaces we considered, $\eps_{(\phi,1)} (J(B_0)/J^2(B_0)) \eps_{(\phi,1)}$ is contained in $\iota_2(e_1) (J(B_0)/J^2(B_0)) \iota_2(e_1)$. We saw that the latter is spanned by the $S_{\mu,\nu}+J^2(B_0)$. But $\eps_{(\phi,1)} S_{\mu,\nu}\eps_{(\phi,1)}=0$ for all choices of $\mu,\nu$.
		\item By  Proposition~\ref{prop emb ai}~\eqref{point formula pi} (and Remark~\ref{remark iota hom}) we have 
		\begin{equation}
		\pi(S_{1,\phi_1}S_{\phi_1,\phi_2}\cdots S_{\prod_{i=1}^{\ell-1}\phi_i,\phi_\ell}) = s_{1,\phi_1}s_{\phi_1,\phi_2}\cdots s_{\prod_{i=1}^{\ell-1}\phi_i,\phi_\ell}\otimes e_1.
		\end{equation}
		By our knowledge of the basis of $A_1$ the right hand side is contained in $J^{\ell + 1}(A_1\otimes_k A_2)$ if and only if $\phi_1=\ldots=\phi_\ell$. Now the assertion follows by Proposition~\ref{prop emb ai}~\eqref{point radical}.
		\item By  Proposition~\ref{prop emb ai}~\eqref{point formula pi} (and Remark~\ref{remark iota hom}) we have 
		\begin{equation}
			\pi(S_{1,\phi}S_{\phi,\phi^{-1}}) = s_{1,\phi}s_{\phi,\phi^{-1}}\otimes e_1,\quad 
			\pi(S_{1,\phi}S_{\phi,\phi^{-1}}T_{1,\zeta}T_{\zeta,\zeta^{-1}}) = s_{1,\phi}s_{\phi,\phi^{-1}}\otimes s_{1,\zeta}s_{\psi,\zeta^{-1}},
		\end{equation}
		and by our knowledge of the bases of $A_1$ and $A_2$ these elements are linearly independent modulo $J^3(A_1\otimes_k A_2)$ and $J^5(A_1\otimes_k A_2)$, respectively. Our claim follows using Proposition~\ref{prop emb ai}~\eqref{point radical}.
		\item By Remark~\ref{remark iota hom} we see that 
		\begin{equation}
		\tilde S^\chi_{\phi} = \iota_1\left(\sum_{g\in L_1} \chi(g^{-1}) s_{1,\phi^g} s_{\phi^g,(\phi^{-1})^g} \right)\iota_2(e_1) \in \iota_1(e_1A_1^\chi e_1)\iota_2(e_1),
		\end{equation}
		that is, $\tilde S^\chi_{\phi}=\iota_1(\tilde s^\chi_{\phi})\iota_2(e_1)$ for some $\tilde s^\chi_{\phi} \in e_1A_1^\chi e_1$. Analogously we have $\tilde T^\eta_{\zeta}=\iota_1(e_1)\iota_2(\tilde t^\eta_{\zeta})$ for some $\tilde t^\eta_{\zeta} \in e_1A_2^\eta e_1$. Hence 
		\begin{equation}
		\tilde S^\chi_{\phi} \tilde T^\eta_{\zeta}=\iota_1(\tilde s^\chi_{\phi})\iota_2(\tilde t^\eta_{\zeta})\quad \textrm{and}\quad   \tilde T^\eta_{\zeta}\tilde S^\chi_{\phi}=\iota_2(\tilde t^\eta_{\zeta})\iota_1(\tilde s^\chi_{\phi}).
		\end{equation}
		The statement now follows directly from Proposition~\ref{prop emb ai}~\eqref{point commutation}.\qedhere
	\end{enumerate}
\end{proof}

\section{Morita equivalences}\label{sec Mor equ}

From now on let $\theta, \theta'\in \Irr(Z)$ be two faithful characters. We keep all other notation from the previous section. The equivalent in $B_0'=B(\theta')_0$ of the various elements given in Definition~\ref{defi eps s} will be denoted with a prime, e.g. $\eps_{(\phi,\psi)}'$, $S_{\psi,\phi}'$, $(\tilde S^\chi _{\phi})'$, and so on. We will show under which conditions $B_0$ and $B_0'$ are Morita equivalent. This will immediately enable us to prove the main theorem of this paper. 

\begin{prop}\label{prop e11 is special}  
	%Let  
	\begin{enumerate}
		\item\label{point dim preserved} A Morita equivalence between $B_0$ and $B'_0$ preserves the dimensions of simple modules. In particular, any such Morita equivalence is afforded by an isomorphism.
		\item An isomorphism $\tau:\ B_0\longrightarrow B'_0$ can be modified by an inner automorphism such that 
		\begin{equation}
		\left\{\sum_{\phi \in \Irr(P_1)} \tau(\eps_{(\phi,1)}), \sum_{\psi \in \Irr(P_2)} \tau(\eps_{(1,\psi)})\right\}=\left\{\sum_{\phi \in \Irr(P_1)} \eps'_{(\phi,1)}, \sum_{\psi \in \Irr(P_2)} \eps'_{(1,\psi)}\right\}.
		\end{equation}
	\end{enumerate}	
\end{prop}
	\begin{proof}
		We will prove this by finding distinguishing Morita invariant properties of the simple modules in $B_0$ and the attached idempotents. By definition we have $\pi(\eps_{(\phi,\psi)})=e_{[\phi]}\otimes e_{[\psi]}$ if $\phi\neq 1$ and $\psi\neq 1$. Using Proposition~\ref{prop emb ai}~\eqref{point radical}~and~\eqref{point ideals}, we get
		\begin{equation}
			\begin{array}{rcll}
			\pi(\eps_{(\phi,\psi)}\xcdot  J(B_0)/J^2(B_0)\xcdot \eps_{(\phi,\psi)})&=&&e_{[\phi]}(J(A_1)/J^2(A_1))e_{[\phi]} \otimes_k e_{[\psi]}(A_2/J(A_2))e_{[\psi]} \\
			 &&\oplus& e_{[\phi]}(A_1/J(A_1))e_{[\phi]} \otimes_k e_{[\psi]}(J(A_2)/J^2(A_2))e_{[\psi]},
			\end{array}
		\end{equation}
		which is clearly non-zero (e.g. $s_{\phi, \phi^{-1}\phi^g}\otimes e_{[\psi]}$ gives a non-trivial on the right hand side for any $1\neq g\in L_1$). In particular the simple modules belonging to the characters of the form $(\phi,\psi)$ all have non-trivial self-extensions. This implies the first assertion since by Lemma~\ref{lemma arrow properties}~\eqref{point ext} the other simples do not have non-trivial self-extensions.
		
		From  Lemma~\ref{lemma arrow properties}~\eqref{point ext}  we already know that the $\Ext^1$ between two simple modules labeled by $(\phi,1)$ and $(\phi',1)$ is one-dimensional if $\phi \neq \phi'\in\Irr(P_1)$. The analogous statement holds for the simples labeled by $(1,\zeta)$ and $(1,\zeta')$. It therefore suffices to show that there are no non-trivial extensions between the simples labeled $(\phi,1)$ and $(1,\zeta)$, where $1\neq \phi \in \Irr(P_1)$ and $1\neq \zeta\in\Irr(P_2)$. The sum of the $\eps_{(\phi,1)}$ for $\phi\neq 1$ is equal to 
		$f_1=\iota_1(1-e_1)\iota_2(e_1)$, and, analogously, the sum of the $\eps_{(1,\zeta)}$ for $\zeta\neq 1$ is equal to $f_2=\iota_1(e_1)\iota_2(1-e_1)$. Note that $f_1$ and $f_2$ are suitable for application of Proposition~\ref{prop emb ai}~\eqref{point ideals}, and we have $\pi(f_1)=(1-e_1)\otimes e_1$ and $\pi(f_2)=e_1\otimes(1-e_1)$. Hence
		\begin{equation}
		    \begin{split}
			\pi(\eps_{(\phi,1)}J(B_0)\eps_{(1,\zeta)})
			\subseteq ((1-e_1)\otimes e_1) J(A_1\otimes_k A_2)(e_1\otimes (1-e_1))\\
			= (1-e_1)J(A_1)e_1\otimes_k e_1 J(A_2)(1-e_1)\subseteq J^2(A_1\otimes_k A_2).
			\end{split}
		\end{equation}
		It follows that $\eps_{(\phi,1)}(J(B_0)/J^2(B_0))\eps_{(1,\zeta)}=0$, that is, there are no non-trivial extensions between the corresponding simple modules.
%		
%		 Lastly, again  using Proposition~\ref{prop emb ai}~\eqref{point radical}~and~\eqref{point ideals}, one sees that $\hat \iota_i(1-e_1)  \xcdot J(B_0)/J^2(B_0) \xcdot \hat \iota_j(1-e_1) = \{0\}$ for $i\neq j$, which implies that a Morita auto-equivalence of $B_0$ either swaps of fixes the sets $\{(\psi,1) \ | \ \psi \in \Irr(P_1) \}$ and $\{(1,\psi) \ | \ \psi \in \Irr(P_2) \}$, which implies second claim.
	\end{proof}

	\begin{remark}\label{remark complete set}
%	Since 
%	$$
%	\begin{array}{rcl}
%	1-(1-\iota_1(e_1))(1-\iota_2(e_1))&=&\iota_1(e_1)+\iota_2(e_1)-\iota_1(e_1)\xcdot \iota_2(e_1)\\\\
%	&=&\displaystyle\sum_{\psi\in \Irr(P_1)}\iota_2(e_\psi)\iota_1(e_1)+\sum_{\xi\in\Irr(P_2)}\iota_1(e_\xi)\iota_2(e_1)-\iota_1(e_1)\xcdot \iota_2(e_1)
%	\end{array}
%	$$ it follows that 
	From Proposition~\ref{prop idempotents} and Proposition~\ref{prop e11 is special}~(2) it follows that
	a $\tau$ as in Proposition~\ref{prop e11 is special}~(2) will, up to an inner automorphism, satisfy either
	\begin{equation}
	 \tau(\eps_{(\phi,1)})=\eps'_{(\sigma(\phi),1)},\quad \textrm{or} \quad \tau(\eps_{(\phi,1)})=\eps'_{(1,\sigma(\phi))},
	\end{equation} 
	for a bijective map $\sigma$ from $\Irr(P_1)$ to either  $\Irr(P_1)$ or $\Irr(P_2)$. 
	The analogous statement holds for the $\eps_{(1,\zeta)}$, and thus, in particular, $\tau(\eps_{(1,1)})=\eps'_{(1,1)}$.
	\end{remark}

	\begin{prop}\label{prop automorphisms rigid}
		Let $\tau:\ B_0\longrightarrow B_0'$ be an isomorphism 
		such that 
		\begin{equation}
			\tau(\eps_{(\phi,1)}) = \eps'_{(\sigma(\phi),1)}\quad\textrm{ for a bijective map $\sigma:\ \Irr(P_1)\longrightarrow \Irr(P_1)$},
		\end{equation}
		and $\sigma(1)=1$. Then $\sigma$ is a group automorphism.
	\end{prop}
	\begin{proof}
		 All we need to show is that $\sigma(\zeta^i)=\sigma(\zeta)^i$ for all $i\geq 0$ for some arbitrary generator $\zeta\in \Irr(P_1)$. For $i< 2$ this is clear. 
		By way of induction we may assume that $\sigma(\zeta^j)=\sigma(\zeta)^j$ for all $j< i$, and $i\geq 2$. 
		Now
		for any $\psi\in \Irr(P_1)$ and $1\neq \phi \in \Irr(P_1)$ we can write (using Lemma~\ref{lemma arrow properties}~\eqref{point ext})
		 \begin{equation}
		 	\tau(S_{\psi,\phi})+J^2(B_0')= c_{\psi,\phi}\xcdot S'_{\sigma(\psi), \sigma(\psi)^{-1}\sigma(\psi\phi)}+ J^2(B_0')\quad\textrm{for some $c_{\psi,\phi}\in k^\times$.}
		 \end{equation}	
		  By applying $\tau$ to the element from  Lemma~\ref{lemma arrow properties}~\eqref{point all psi equal} it follows that 
		\begin{equation}\label{eqn rrksks}	
			S'_{\sigma(\psi), \sigma(\psi)^{-1}\sigma(\psi\phi)}S'_{\sigma(\psi\phi), \sigma(\psi\phi)^{-1}\sigma(\psi\phi^2)} \cdots  S'_{\sigma(\psi\phi^{\ell-1}), \sigma(\psi\phi^{\ell-1})^{-1}\sigma(\psi\phi^\ell)} \in J^{\ell+1}(B_0'),
		\end{equation}
		again for any $\psi\in \Irr(P_1)$ and $1\neq \phi \in \Irr(P_1)$. By Lemma~\ref{lemma arrow properties}~\eqref{point all psi equal} all second indices occurring in \eqref{eqn rrksks}, that is, all $\sigma(\psi\phi^q)^{-1}\sigma(\psi\phi^{q+1})$ for $0\leq q <\ell$, must be equal for the element to be contained in $J^{\ell+1}(B_0')$. In particular, if we specialise $\psi=\zeta^{i-2}$, $\phi=\zeta$ and look at $q=0$ and $q=1$, we get
		\begin{equation}
			\sigma(\zeta^{i-2})^{-1}\sigma(\zeta^{i-1})=\sigma(\zeta^{i-1})^{-1}\sigma(\zeta^{i}).
		\end{equation}
		The left hand side is equal to $\sigma(\zeta)$ by the induction hypothesis, which implies that 
		$\sigma(\zeta^i)=\sigma(\zeta^{i-1})\xcdot \sigma(\zeta)=\sigma(\zeta)^i$, which completes the induction step.
	\end{proof}

	Of course the analogue of the above statement with $(\phi,1)$ swapped for $(1,\zeta)$ holds as well.

\begin{prop}\label{prop non isom}
	The block $B({\theta})$ is Morita equivalent to $B({\theta'})$ if and only if $\theta'=\theta^{\pm 1}$.
\end{prop}

\begin{proof}
	We first note that if $\theta'=\theta^{\pm 1}$, then, by Lemma \ref{prop isoms}~(2), $B$ is Morita equivalent to $B'$.	
	Conversely, suppose $B$ is Morita equivalent to $B'$. Of course this implies that $B_0$ and $B'_0$ are Morita equivalent. Moreover, by Proposition \ref{prop e11 is special}~\eqref{point dim preserved}, any such Morita equivalence must preserve the dimensions of the simple modules and so we may assume that it is induced by a $k$-algebra isomorphism $\tau:B_0\to B'_0$.
	
	By Remark~\ref{remark complete set} we may assume, after pre-composing with an inner automorphism and the isomorphism from Lemma~\ref{prop isoms}~(2) (in which case we replace $\theta$ by $\theta^{-1}$), that 
	$\tau(\eps_{(\phi,1)}) = \eps'_{(\sigma_1(\phi),1)}$ and $\tau(\eps_{(1,\zeta)}) = \eps'_{(1,\sigma_2(\zeta))}$ for maps $\sigma_i:\ \Irr(P_i)\longrightarrow \Irr(P_i)$. Furthermore, by Proposition~\ref{prop automorphisms rigid} we may assume that the $\sigma_i$ are group automorphisms of $\Irr(P_i)$. Certainly every group automorphism of $\Irr(P_i)$ is induced by one of $P_i$ and so, possibly after pre-composing $\tau$ with an automorphism as in Lemma \ref{prop isoms}~(1), we may assume that both $\sigma_i$ are the identity. That is, we may assume
	\begin{equation}\label{eqn idempot preserved}
		\tau(\eps_{(\phi,1)})=\eps'_{(\phi,1)},\quad \tau(\eps_{(1,\zeta)})=\eps'_{(1,\zeta)} \quad \textrm{for any $\phi\in\Irr(P_1)$ and $\zeta\in\Irr(P_2)$}.
	\end{equation}
	From now on we fix a $\phi\in \Irr(P_1)\setminus\{1\}$ and $\zeta\in \Irr(P_2)\setminus\{1\}$ and define the spaces
	\begin{align*}
	\boldsymbol{S}_\phi&=\langle \tilde{S}_\phi^\chi|\chi\in\Irr(L_1)\rangle_k=\langle S_{1,\phi^h}S_{\phi^h,(\phi^h)^{-1}}|h\in L_1\rangle_k,\\
	\boldsymbol{T}_\zeta&=\langle \tilde{T}_\zeta^\eta|\eta\in\Irr(L_2)\rangle_k=\langle T_{1,\zeta^h}T_{\zeta^h,(\zeta^h)^{-1}}|h\in L_2\rangle_k.
	\end{align*}
	The first assertion of Lemma~\ref{lemma arrow properties}~\eqref{point lin indep} gives that $\boldsymbol{S}_\phi$ and $\boldsymbol{T}_\zeta$ both have dimension $r$ modulo $J^3(B_0)$, with bases $(\tilde{S}_\phi^\chi+J^3(B_0))_{\chi\in\Irr(L_1)}$ and $(\tilde{T}_\zeta^\eta+J^3(B_0))_{\eta\in\Irr(L_2)}$ respectively. From the second assertion of Lemma~\ref{lemma arrow properties}~\eqref{point lin indep} and Lemma~\ref{lemma arrow properties}~\eqref{point schi comm}, we get that
	\begin{equation}\label{algn comm S}
	\begin{array}{rl}
	\{x\in\boldsymbol{S}_\phi+J^3(B_0)|&xy-yx\in J^5(B_0)\text{ for all }y\in \boldsymbol{T}_\zeta+J^3(B_0)\}\\
	=&\langle \tilde{S}_\phi^{1}\rangle_k+J^3(B_0)
	\end{array}
	\end{equation}
	and similarly
	\begin{equation}\label{algn comm T}
	\begin{array}{rl}
	\{x\in\boldsymbol{T}_\zeta+J^3(B_0)|&xy-yx\in J^5(B_0)\text{ for all }y\in \boldsymbol{S}_\phi+J^3(B_0)\}\\
	=&\langle \tilde{T}_\zeta^{1}\rangle_k+J^3(B_0).
	\end{array}
	\end{equation}
	Of course, the analogous assertions for the algebra $B'_0$ hold as well. 
	
	Equation \eqref{eqn idempot preserved} combined with Lemma~\ref{lemma arrow properties}~\eqref{point ext} implies that
	\begin{align}\label{algn:ST mod J2}
	\begin{split}
	\tau(\langle S_{\mu,\nu}\rangle_k)+J^2(B_0)&=\langle S'_{\mu,\nu}\rangle_k+J^2(B'_0),\\
	\tau(\langle T_{\gamma,\delta}\rangle_k)+J^2(B_0)&=\langle T'_{\gamma,\delta}\rangle_k+J^2(B'_0),
	\end{split}
	\end{align}
	for all $\mu,\nu\in\Irr(P_1)$ and $\gamma,\delta\in\Irr(P_2)$, $\nu\neq 1$ and $\delta \neq 1$. Therefore, by Lemma~\ref{lemma arrow properties}~\eqref{point lin indep},
	\begin{align*}
	\tau(S_{1,\phi^g}S_{\phi^g,(\phi^g)^{-1}})+J^3(B_0)&=u_{g}\xcdot  S'_{1,\phi^g}S'_{\phi^g,(\phi^g)^{-1}}+J^3(B'_0),\\
	\tau(T_{1,\zeta^h}T_{\zeta^h,(\zeta^h)^{-1}})+J^3(B_0)&=v_{h}\xcdot T'_{1,\zeta^h}T'_{\zeta^h,(\zeta^h)^{-1}}+J^3(B'_0),	
	\end{align*}
	for all $g\in L_1$, $h\in L_2$ and uniquely determined $u_{g},v_{h}\in k^\times$. Furthermore, \eqref{algn comm S} and \eqref{algn comm T}	 give 
	\begin{align*}
	\tau(\langle \tilde S_\phi^{1}\rangle_k)+J^3(B_0)=\langle (\tilde S_\phi^{1})'\rangle_k+J^3(B'_0),\\
	\tau(\langle \tilde T_\zeta^{1}\rangle_k)+J^3(B_0)=\langle (\tilde T_\zeta^{1})'\rangle_k+J^3(B'_0).
	\end{align*}
	Therefore, all the $u_{g}$'s are equal, say $u$. Similarly we set $v$ to be the common value of the $v_{h}$'s. In particular,
	\begin{align*}
	\tau(\tilde{S}_\phi^\chi)+J^3(B_0)&=u\xcdot (\tilde{S}_\phi^\chi)'+J^3(B'_0),\\
	\tau(\tilde{T}_\zeta^\eta)+J^3(B_0)&=v\xcdot (\tilde{T}_\zeta^\eta)'+J^3(B'_0),
	\end{align*}
	for all $\chi\in\Irr(L_1)$ and $\eta\in\Irr(L_2)$. From Lemma~\ref{lemma arrow properties}~\eqref{point schi comm} we get the identities
	\begin{equation}
	\tilde S^\chi_{\phi} \tilde T^\eta_{\zeta} = \theta([h^\theta_{\eta,2}, h^\theta_{\chi,1}]) \xcdot \tilde T^\eta_{\zeta} \tilde S^\chi_{\phi}\quad\textrm{and}\quad(\tilde S^\chi_{\phi})' (\tilde T^\eta_{\zeta})' = \theta'([h^{\theta'}_{\eta,2}, h^{\theta'}_{\chi,1}]) \xcdot (\tilde T^\eta_{\zeta})' (\tilde S^\chi_{\phi})'
	\end{equation}
	Using the second assertion of Lemma~\ref{lemma arrow properties}~\eqref{point lin indep} we can apply $\tau$ to the first and compare to the second modulo $J^5(B_0')$. That gives
	\begin{equation}
	\eta(h^\theta_{\chi,1})=\theta([h^\theta_{\eta,2}, h^\theta_{\chi,1}])=\theta'([h^{\theta'}_{\eta,2} h^{\theta'}_{\chi,1}])=\eta(h_{\chi,1}^{\theta'}),
	\end{equation} 
	for all $\chi\in\Irr(L_1)$ and $\eta\in\Irr(L_2)$. Therefore, $h^\theta_{\chi,1}=h^{\theta'}_{\chi,1}$, for all $\chi\in\Irr(L_1)$. Finally, since $[L_1,L_2]=Z$, it follows from Definition~\ref{def h chi} that $\theta=\theta'$.
\end{proof}
	
	\begin{thm}\label{thm:main}
		Let $\ell$ be a prime and $n\in\N$. Then there exists an $\ell$-block $B$ of $kG$, for a finite group $G$, such that $\mf(B)=n$.
	\end{thm}
	\begin{proof}
	    Of course, we are setting to $G$ and $B$ to be as in the rest of the article. We just need an appropriate choice of $p$ and $r$ and $\theta$.
	
	    Set $r=\ell^n+1$. By the Dirichlet prime number theorem, we can set $p$ to be a  prime congruent to $1$ modulo $\ell$ and modulo $r$. Set $\theta$ to be any faithful, linear character of $Z$. Note that $B(\theta)^{(\ell^m)}=B(\theta^{\ell^m})$, for all $m\in\mathbb{N}$, and so, by Proposition~\ref{prop non isom}, $\mf(B)$ is the smallest $m\in\mathbb{N}$, such that $\theta^{\ell^m}=\theta^{\pm1}$ or equivalently that $r|(\ell^m\pm 1)$. It is now clear that $\mf(B)=n$.
	\end{proof}

	\subsection*{Acknowledgements} The authors were supported by EPSRC grants 	
	EP/T004592/1 and EP/T004606/1. We would also like to thank Charles Eaton and Radha Kessar for helpful discussions on the subject matter of this paper.

	\bibliographystyle{alpha}

\end{document}